\documentclass[12pt]{article}
\usepackage[cp1250]{inputenc}
\usepackage{amssymb}
\usepackage{amsmath}
\usepackage{indentfirst}
\usepackage{mathabx} 
\usepackage{amsthm} 

\usepackage{MnSymbol} 

\usepackage{graphicx}
\usepackage{cmll}
\usepackage{geometry} 
\usepackage{enumitem} 

\theoremstyle{plain}
\newtheorem{theorem}{Theorem}
\newtheorem{lemma}{Lemma}
\newtheorem{proposition}{Proposition}
\newtheorem{corollary}{Corollary}
\newtheorem{definition}{Definition}
\theoremstyle{definition}
\newtheorem{example}{Example}

\def\Aut{\textrm{\textup{Aut}}}
\def\Cay{\textrm{\textup{Cay}}}

\begin{document}
\newgeometry{tmargin=3cm, bmargin=3cm, lmargin=2.5cm, rmargin=2.5cm}
\pagestyle{empty}

\begin{center}
\begin{large}
\textbf{The Girth of Cayley graphs of Sylow 2-subgroups of~symmetric~groups~$S_{2^n}$ on diagonal bases}\\
\end{large}
Bartłomiej Pawlik\\
\end{center}

\begin{small}

\bigskip

\textbf{Abstract.} A diagonal base of a Sylow 2-subgroup $P_n(2)$ of symmetric group $S_{2^n}$ is a minimal generating set of this subgroup consisting of elements with only one non-zero coordinate in the~polynomial representation. For different diagonal bases Cayley graphs of $P_n(2)$ may have different girths (i.e.~minimal lengths of cycles) and thus be non-isomorphic. In presented paper all possible values of girths of Cayley graphs of $P_n(2)$ on diagonal bases are calculated. A criterion for whenever such Cayley graph has girth equal to 4 is presented. A lower bound for the number of different non-isomorphic Cayley graphs of $P_n(2)$ on diagonal bases is proposed.
\end{small}

\section{Introduction.}


\bigskip

Let $P_n(p)$ be a Sylow $p$-subgroup of symmetric group $S_{p^n}$.  A minimal generating set of group G, i.e. a generating set from which none of the elements can be removed, is called \textit{a base} of group G. All bases of $P_n(p)$ contain exactly $n$ elements (see \cite{paw2016}, \cite{sus2009}).

It is known that given a generating set $S$ of group $G$, the Cayley graph $\Cay(G,S)$ of $G$ with respect to $S$ is a simple directed graph with vertex set $V(\Cay(G,S)) = G$ and the set of edges
$$E(\Cay(G,S))=\{(uv)|\,\exists(s\in S):\,v=us\}.$$
If additionally $S$ is symmetric (i.e. $S=S^{-1}$) then $\Cay(G,S)$ may be considered undirected. For two different generating sets $S$ and $S'$ of group $G$, the resulting Cayley graphs $\Cay(G,S)$ and $\Cay(G,S')$ may differ and be non-isomorphic. Thus a fundamental question in the~study of Cayley graphs of groups is the graph isomorphism problem. It has been widely studied by many authors (see e.g. \cite{ bab1978, dob1995, kon2008, sus2009}, and for the survey - \cite{li2002}). On approach to address the~graph isomorphism problem is to study graph features, which are invariant to isomorphisms, e.g. diameter, chromatic number, clique number, independence number (see e.g. \cite{bam2014, gan2013, gre2017, klo2007}) or girth (\cite{gam2009, loz2011}). Let us recall the result of \cite{gam2009}] which states that the girth $g_n$ of random Cayley graphs of Sylow $p$-subgroups $P_n(p)$ of group $S_{p^n}$ satisfies 
$$(1-o(1))\beta\log\log|P_n(p)|\leq g_n\leq(1+o(1))(\log|P_n(p)|)^\alpha,$$
with $n\rightarrow\infty$, where $\alpha<1$ is a constant depending on $p$ only, and $\beta$ depends on $p$ and $d$ (degree of regularity of given graph). This paper presents an example of family of Cayley graphs of Sylow 2-subgroups $P_n(2)$ of group $S_{2^n}$ which, for all $n$, have a girth at most 8.

In this paper we investigate Cayley graphs of $P_n(p)$ defined on the generating sets of certain type, the so called diagonal beses (see \cite{paw2016, sus2009}). Since every diagonal base consists of involutions, the respective Cayley graphs are simple and undirected.

 It is known that group $P_n(2)$ is isomorphic to the group $\Aut(T_2^n)$ of automorphisms of a~rooted binary tree~$T^n_2$ of height $n$ (see e.g. \cite{bie2015}). The base $D=\{D_1,\ldots,D_n\}$ of $P_n(2)$ such that element $D_i$ for $i=1,\ldots,n,$ acts only on the $i$-th level of $T^n_2$ is called a \textit{diagonal base}. \textit{The girth} of given graph $\Gamma$ is the length of the minimal cycle of this graph. The Cayley graph $\Cay(P_n,D)$ of group $P_2(2)$ on any diagonal base $D$ is a $C_8$ cycle and it is isomorphic to an induced subgraph of any Cayley graph of $P_n(2)$ on diagonal base $D$ for $n>2$ (see \cite{paw2016}). Thus every graph $\Cay(P_n(2),D)$, where $D$ is a~diagonal base and $n\geq2$, contains a~cycle of length~8, however, not necessarily it is minimal. For given $n\geq3$ there exist Cayley graphs of $P_n(2)$ on diagonal bases which have smaller minimal cycles, and thus not every two such Cayley graphs of $P_n(2)$ are isomorphic. In this paper we classify Cayley graphs of $P_n(2)$ on diagonal bases with respect to their girths in the following
 
\begin{theorem}[Main Theorem]\label{maintheorem}
 Let $D=\{D_1,\ldots,D_n\}$ be a diagonal base of $P_n(2)$.
 \begin{enumerate}
 \item The girth of $\Cay(P_n(2),D)$ is either 4, 6 or 8.
 \item For $n\geq3$ the girth of $\Cay(P_n(2),D)$ is equal to 4 if and only if $D$ contains two commuting elements, i.e. $D$ satisfies condition $(\star)$ stated in Section~3.  
 \item For all $n\geq2$ there exists $D$ such that the girth of $\Cay(P_n(2),D)$ is equal to 8.
 \item For all $n\geq5$ there exists $D$ such that the girth of $\Cay(P_n(2),D)$ is equal to 6. For $n<5$ there are no such diagonal bases.
 \end{enumerate}
 \end{theorem}

 By the result of \cite{paw2016} the upper bound for the number of non-isomorphic Cayley graphs of $P_n(2)$ is equal to $\displaystyle2^{2^n-2n}$. We apply point 2 of Theorem \ref{maintheorem} to obtain a lower bound: we show that the~number of non-isomorphic Cayley graphs of $P_n(2)$ on diagonal bases is at least $\displaystyle2^{n-2}$.
 
The paper is organized as follows. In Section 2 we recall some basic facts about groups $P_n(2)$, their polynomial representation and diagonal bases. Also in this Section we recall some basic facts and properties of Cayley graphs. In Section 3 we prove the Main Theorem. Application of the Main Theorem to estimate the lower bound of the number of non-isomorphic Cayley graphs of $P_n(2)$ on diagonal bases is proposed in Section 4.

\section{Preliminaries.}

Let $X_n$ be a sequence of variables $x_1,\ldots,x_n$.

It is well known (see e.g. \cite{kal1948}) that the Sylow $p$-subgroups of group $S_{p^n}$ are isomorphic with an~$n$-iterated wreath product of cyclic groups $C_p$, namely
$$P_n(p)\cong \underbrace{C_p\wr C_p\wr\ldots\wr C_p}_{n}.$$
Group $C_p$ is isomorphic to the additive group $\mathbb{Z}_p$, thus every element $f\in P_n(p)$ may be written as
\begin{equation}\label{eq:f}
f=[f_1,f_2(X_1),f_3(X_2),\ldots,f_n(X_{n-1})],
\end{equation}
where $f_1\in\mathbb{Z}_p$ and $f_i:\mathbb{Z}_p^{i-1}\rightarrow\mathbb{Z}_p$ for $i=2,\ldots,n$ are reduced polynomials from the quotient ring $\mathbb{Z}_p[X_i]/\langle x_1^p-x_1,\ldots,x_i^p-x_i\rangle$. We call such element $f$ a \textit{tableau}.

For tableaux $f,g\in P_n(p)$, where $f$ has form (\ref{eq:f}) and $$g=[g_1,g_2(X_1),g_3(X_2),\ldots,g_n(X_{n-1})]$$ the product $fg$ has the form
\begin{align*}
fg=\big[&f_1+g_1,\,f_2(X_1)+g_2(x_1+f_1),\,\ldots,\\
&f_n(X_{n-1})+g_n(x_1+f_1,x_2+f_2(X_1),\ldots,x_{n-1}+f_{n-1}(X_{n-2}))\big],
\end{align*}
and the inverse $f^{-1}$ of an element $f$ has the form
\begin{align*}
f^{-1}=\Big[&-f_1,-f_2(x_1-f_1),\ldots,\\
&-f_n\big(x_1-f_1,x_2-f_2(x_1-f_1),\ldots,x_{n-1}-f_{n-1}(x_1-f_1,\ldots)\big)\Big].
\end{align*}

Let $[f]_i$ be the $i$-th coordinate of tableau $f$, i.e. $[f]_1=f_1$ and $[f]_i=f_i(X_{i-1})$ for $i=2,\ldots,n$, and $(f)_i$ be a tableau consisting of first $i$ coordinates of $f$, i.e. $(f)_i=[f_1,\ldots,f_i(X_{i-1})]$.

\smallskip

Further in this paper tableau $[0,\ldots,0]$ will be simply denoted as $0$.

\bigskip

Let $\overline{x_n}=x_1\cdot x_2\cdot\ldots\cdot x_n$ and $$\overline{x_n}/x_i=x_1\cdot x_2\cdot\ldots\cdot x_{i-1}\cdot x_{i+1}\cdot\ldots\cdot x_n.$$

\smallskip

A diagonal base of $P_n(2)$, defined in the Introduction consists of these and only these automorphisms of $T_2^n$, whose action is nontrivial on a unique level of the tree. In terms of  polynomial representation, diagonal bases may be defined as following.

\begin{definition}
Base $D=\{D_1,\ldots,D_n\}$ of $P_n(2)$ is called diagonal if $[D_i]_j=0$ for any $i$ ($1\leq i\leq n$) and $j\neq i$.
\end{definition}

It is known (see e.g. \cite{kal1948}) that in every base $D$ of $P_n(2)$ for every $i$ there exists a tableau $D'\in D$ which contains a monomial $\overline{x_{i-1}}$ on $i$-th coordinate. Thus, the nonzero coordinates of elements of diagonal base $D=\{D_1,\ldots,D_n\}$ have form $[D_1]_1=1$ and $[D_i]_i=d_i(X_{i-1})$, where $d_i$ contains monomial $\overline{x_{i-1}}$ for every $i=2,\ldots,n$.

We also note that every element of diagonal base $D=\{D_1,\ldots,D_n\}$ is an involution, i.e. $D_i^2=0$ for $i=1,\ldots,n$.

\bigskip

\bigskip

For any graph $\Gamma$ let $V(\Gamma)$ be a vertex set of $\Gamma$ and $E(\Gamma)$ be the set of edges of $\Gamma$.

\bigskip

Let $S=\{s_1,\ldots,s_n\}$ be a generating set of a finite group $G$. The Cayley graph $\Cay(G,S)$ contains a cycle if and only if for some $g\in G$ there exists an irreducible product $s_{i_1}s_{i_2}\ldots s_{i_k}$, where $s_{i_j}\in S$ for every $j=1,\ldots,k$, such that 
$$gs_{i_1}s_{i_2}\ldots s_{i_k}=g.$$
The above equality is equivalent to
$$s_{i_1}s_{i_2}\ldots s_{i_k}=0.$$
Thus there holds the following
\begin{lemma}\label{l:pod}
If $S=\{s_1,\ldots,s_n\}$ is a generating set of a finite group $G$ and for some generators $s_{i_1},s_{i_2},\ldots,s_{i_k}$ the product  $s_{i_1}s_{i_2}\ldots s_{i_k}$ is irreducible and equal to 0, then $\Cay(G,S)$ contains a cycle of length $k$.
\end{lemma}

If $S=S^{-1}$, then $\Cay(G,S)$ is an undirected graph (i.e. if $s$ is an involution and if $u=vs$, then also $v=us$). Since every element of a diagonal base of $P_n(2)$ is an involution, the corresponding Cayley graphs of $P_n(2)$ may be considered as undirected.

\bigskip

Let us recall some graph theory facts, which will be useful in the last Section of this paper.

\smallskip

For any finite connected graph $\Gamma$ we define a natural metric $\sigma$ on the~set of vertices $V(\Gamma)$ in the~following way: $\sigma(u,v)=n$ if and only if the shortest way from vertex $u$ to vertex $v$ contains exactly $n$~edges (we assume that $\sigma(v,v)=0$ for every vertex $v$). We define a special type of induced subgraphs of a graph $\Gamma$.

\begin{definition}
Let $B(v,r)$ be the induced subgraph of given graph $\Gamma$ such that $v\in V(\Gamma)$, $r$~is fixed positive integer and $$V(B)=\{u\in V(\Gamma)|\,\sigma(u,v)\leq r\}.$$
We call $B(v,r)$ a ball with center $v$ and radius $r$.
\end{definition}

\section{Proof of the Main Theorem.}

Let $D=\{D_1,\ldots,D_n\}$ be a diagonal base of $P_n(2)$. Every $D_i$ ($i\geq2$) contains the monomial $\overline{x_{i-1}}$, thus if element $D_i$ appears in the product $D_{j_1}D_{j_2}\ldots D_{j_k}$ an odd number of times, then
$$[D_{j_1}D_{j_2}\ldots D_{j_k}]_i=\overline{x_{i-1}}+f(X_{i-1}),$$
where $f$ do not contain the monomial $\overline{x_{i-1}}$. Thus the necessary condition for the~equation
$$[D_{j_1}D_{j_2}\ldots D_{j_k}]_i=0$$
to hold is that the element $D_i$ appears in this product an even number of times. Hence from Lemma~\ref{l:pod} we get that every cycle of a graph $\Cay(P_n(2),D)$ has even length. Thus the minimal possible girth is equal to 4. As we stated in Introduction, every such graph for $n\geq2$, contains a cycle of length 8, so the girth is at most equal to 8. Let us investigate whenever such graph have lower girth.

\bigskip

Firstly, let us investigate for which diagonal bases of $P_n(2)$ there exists a cycle of length~4. The Graph $\Cay(P_2(2),D)$ is isomorphic to cyclic graph $C_8$, so it do not contain cycles of length~4.

\bigskip

For $n\geq3$ we say that the diagonal base $D=\{D_1,\ldots,D_n\}$ of group $P_n(2)$ satisfies condition~$(\star)$ if and only if for some $D_i$ ($2\leq i<n$) there exists $D_j$ $(i<j\leq n)$ of the form
$$[D_j]_j=\overline{x_{j-1}}+x_ig(X_{j-1})+h(X_{j-1}),$$
where $g$ and $h$ do not contain variable $x_i$ and $g$ do not contain a monomial $\overline{x_{j-1}}/x_i$, such that
$$g(X_{j-1})[D_i]_i=\alpha\cdot\overline{x_{j-1}}/x_i,$$
where
$$\alpha=
\left\{\begin{array}{ll}
0,&\mbox{ if }[D_i]_i\mbox{ contains even number of nonzero monomials,}\\
1,&\mbox{ if }[D_i]_i\mbox{ contains odd number of nonzero monomials.}
\end{array}\right.$$

\begin{proposition}\label{prop:1}
Let $D=\{D_1,\ldots,D_n\}$ be the diagonal base of $P_n(2)$, where $n\geq3$. The Cayley graph $\Cay(P_n(2),D)$ contains a cycle of length 4 if and only if $D$ satisfies condition~$(\star)$.
\end{proposition}

\begin{proof}

From Lemma \ref{l:pod}, the existence of cycle of length 4 in $\Cay(P_n(2),D)$ is equivalent to the existence of solution of equation

$$(D_iD_j)^2=0\ \ \ \ \mbox{ or }\ \ \ \ (D_jD_i)^2=0$$

for some $i,j,$ such that $i<j\leq n$.
The above equalities are equivalent and hold whenever $D_j$ and $D_i$ commute. Thus further we will only consider the case $(D_jD_i)^2=0$.

$D_jD_i$ is a tableau with only nonzero elements on $i$-th and $j$-th coordinates, namely
$$[D_jD_i]_i=[D_i]_i\ \ \ \mbox{ and }\ \ \ [D_jD_i]_j=[D_j]_j.$$

Thus $(D_jD_i)^2=0$ if and only if $$[(D_jD_i)^2]_i=0\ \ \ \mbox{ and }\ \ \ [(D_jD_i)^2]_j=0.$$
The first of the above equalities is always true (simple checking), so we have to investigate only the second one.

\smallskip

Firstly, let us assume that $i=1$. Of course $$D_1=[1,0,\ldots,0]$$ for every diagonal base $D$. Let $$[D_j]_j=\overline{x_{j-1}}+x_1g(X_{i-1})+h(X_{i-1}),$$ where $g$ and $h$ do not contain variable $x_1$ and $g$ do not contain the monomial $\overline{x_{j-1}}/x_1$. Thus
\begin{align*}
[(D_jD_1)^2]_j=&\,\overline{x_{j-1}}+x_1g(X_{i-1})+h(X_{i-1})+(\overline{x_{j-1}}/x_1)(x_1+1)+(x_1+1)g(X_{i-1})+h(X_{i-1})=\\
=&\,\overline{x_{j-1}}/x_1+g(X_{j-1}).
\end{align*}
But $g$ do not contain monomial $\overline{x_{j-1}}/x_1$, and so $\overline{x_{j-1}}/x_1+g(X_{j-1})\neq0$.

Hence the element $D_1$ do not commute with any other element of $D$, and thus it do not belong to any cycle of length 4 in the graph $\Cay(P_n(2),D)$. 

\smallskip

Let us now investigate the general case for $1<i<j\leq n$.

Let
$$[D_i]_i=\overline{x_{i-1}}+f(X_{i-1}),$$
where $f$ do not contain a monomial $\overline{x_{i-1}}$ and let
$$[D_j]_j=\overline{x_{j-1}}+x_ig(X_{j-1})+h(X_{j-1}),$$
where polynomials $g$ and $h$ do not contain variable $x_i$ and $g$ do not contain a monomial $\overline{x_{j-1}}/x_i$.

Then
\begin{align*}
[(D_jD_i)^2]_j=&\,\overline{x_{j-1}}+x_ig(X_{j-1})+h(X_{j-1})+\\
&\,+(\overline{x_{j-1}}/x_i)(x_i+\overline{x_{i-1}}+f(X_{i-1}))+\\
&\,+(x_i+\overline{x_{i-1}}+f(X_{i-1}))g(X_{j-1})+h(X_{j-1})=\\
=&\,(\overline{x_{j-1}}/x_i)(1+f(X_{i-1}))+g(X_{j-1})(\overline{x_{i-1}}+f(X_{i-1})).
\end{align*}
The product $(\overline{x_{j-1}}/x_i)(1+f(X_{i-1}))$ is equal to 0 or $\overline{x_{j-1}}/x_i$ depending of number of nonzero monomials in $f$. Let us check both cases:

\begin{enumerate}

\item If the number of nonzero monomials on $f$ is odd, then $$(\overline{x_{j-1}}/x_i)(1+f(X_{i-1}))=0$$ and thus
$$[(D_jD_i)^2]_j=g(X_{j-1})(\overline{x_{i-1}}+f(X_{i-1}))=g(X_{j-1})[D_i]_i.$$
Hence in this case $D_j$ and $D_i$ commute if and only if $g(X_{j-1})[D_i]_i=0$.

\item If the number of nonzero monomials in $f$ is even, then $$(\overline{x_{j-1}}/x_i)(1+f(X_{i-1}))=\overline{x_{j-1}}/x_i$$ and thus
$$[(D_jD_i)^2]_j=\overline{x_{j-1}}/x_i+g(X_{j-1})(\overline{x_{i-1}}+f(X_{i-1}))=\overline{x_{j-1}}/x_i+g(X_{j-1})[D_i]_i.$$
Hence in this case $D_j$ and $D_i$ commute if and only if $g(X_{j-1})[D_i]_i=\overline{x_{j-1}}/x_i$.
\end{enumerate}
\end{proof}

From Proposition \ref{prop:1} we get some natural examples of Cayley graphs of $P_n(2)$ with/without cycles of length 4:

\begin{example}\label{e:1}
 \begin{enumerate}
  \item If $[D_i]_i=\overline{x_{i-1}}$ for every $i$ such that $1<i\leq n$, then $\Cay(P_n(2),D)$ do not contain a cycle of length 4.
  \item If $[D_i]_i=\overline{x_{i-1}}+a_i$ for every $i$ such that $1<i<n$ and $$\max\{a_i:1\leq i\leq n\}=1,$$ then $\Cay(P_n(2),D)$ contains a cycle of length 4.
 \end{enumerate}
\end{example} 

In further investigation of the girths of $P_n(2)$ we will use the  following

\begin{proposition}\label{prop:2}
Let $D=\{D_1,\ldots D_n\}$ be a diagonal base of group $P_n(2)$ such that graph $\Cay(P_n(2),D)$ do not contain cycle of length 4, and let $D_{e_1},\ldots,D_{e_k}$ be those elements of~$D$ which contain an even number of monomials. Let $\overline{D}=\{\overline{D}_1,\ldots,\overline{D}_{n+1}\}$ be a diagonal base of $P_{n+1}(2)$ such that $(\overline{D}_i)_n=D_i$ for $i=1,\ldots,n$, and
$$[\overline{D}_{n+1}]_{n+1}=\overline{x_n}+\sum\limits_{j=1}^kx_{e_j}.$$
Then $\Cay(P_{n+1}(2),\overline{D})$ do not contain a cycle of length 4.
\end{proposition}

\begin{proof}
Graph $\Cay(P_{n+1},\overline{D})$ contains a cycle of length 4 if there exist different generators $\overline{D}_i,\overline{D}_j\in\overline{D}$ which commute. The graph $\Cay(P_n(2),D)$ do not contain a cycle of length 4, so none of generators $\overline{D}_1,\overline{D}_2,\ldots,\overline{D}_n$ commute. Thus we only have to show that for every $i=1,\ldots,n$ generator $\overline{D}_{n+1}$ do not commute with $\overline{D}_i$.

We have shown that the generator $\overline{D}_1$ do not commute with any other generator, so we may assume that $i\geq2$.

Let
$$[\overline{D}_i]_i=\overline{x_{i-1}}+d_i(X_{i-1}),$$
where $d_i$ do not contain a monomial $\overline{x_{i-1}}$.

Let us consider two cases:
\begin{enumerate}
\item $\overline{D}_i$ contains an even number of monomials.

In this case
\begin{align*}
[(\overline{D}_{n+1}\overline{D}_{i})^2]_{n+1}&=\overline{x_n}+\sum\limits_{j=1}^kx_{e_j}+(\overline{x_n}/x_i)(x_i+\overline{x_{i-1}}+d_i(X_{i-1}))+\sum\limits_{j=1,e_j\neq i}^kx_{e_j}+x_i+\overline{x_{i-1}}+d_i(X_{i-1})=\\
&=(\overline{x_n}/x_i)(1+d_i(X_{i-1}))+d_i(X_{i-1}).
\end{align*}
$d_i$ contains an odd number of nonzero monomials, so $d_i\neq0$ and $(\overline{x_n}/x_i)(1+d_i(X_{i-1}))=0$.

Thus $[(\overline{D}_{n+1}\overline{D}_i)^2]_{n+1}=d_i(X_{i-1})\neq0$.

\item $\overline{D}_i$ contains an odd number of monomials.

In this case
\begin{align*}
[(\overline{D}_{n+1}\overline{D}_{i})^2]_{n+1}&=\overline{x_n}+\sum\limits_{j=1}^kx_{e_j}+(\overline{x_n}/x_i)(x_i+\overline{x_{i-1}}+d_i(X_{i-1}))+\sum\limits_{j=1}^kx_{e_j}=\\
&=(\overline{x_n}/x_i)(1+d_i(X_{i-1})).
\end{align*}
$d_i$ contains an even number of monomials, so $(\overline{x_n}/x_i)(1+d_i(X_{i-1}))\neq0$ and hence $[(\overline{D}_{n+1}\overline{D}_i)^2]_{n+1}\neq0$.
\end{enumerate}
Thus none of elements of base $\overline{D}$ commutes, and so graph $\Cay(P_{n+1}(2),\overline{D})$ do not contain a~cycle of length 4.
\end{proof}

The Cayley graph $\Cay(P_n(2),D)$ on diagonal base $D=\{D_1,\ldots,D_n\}$ contains a cycle of length 6 if and only if there exists generators $D_i,\,D_j,\,D_k$, where $i<j<k$ such that one of the following equalities holds:
\begin{align*}
(D_kD_jD_i)^2&=0\\
(D_kD_iD_j)^2&=0\\
D_kD_jD_iD_jD_kD_i&=0\\
D_kD_jD_kD_iD_jD_i&=0\\
D_kD_jD_iD_kD_iD_j&=0
\end{align*}

Let us notice that only the last equation do not imply the existence of cycle of length 4 (for example if $D_kD_jD_iD_jD_kD_i=0$, then also $D_jD_iD_jD_i=0$). Thus $\Cay(P_n(2),D)$ have a~minimal cycle of length 6 iff none of the elements of $D$ commute and if for some generators $D_i,\,D_j,\,D_k$ ($i<j<k$) the following equality holds:
\begin{equation}\label{eq:1}
D_kD_jD_iD_kD_iD_j=0. 
\end{equation}

Before we state the next proposition, we will need the following technical Lemma:

\begin{lemma}\label{l:c6}
If $D=\{D_1,\ldots,D_n\}$ is a diagonal base of $P_n(2)$ where $n\geq3$, then
$$D_kD_jD_1D_kD_1D_j\neq0$$
for every $j,k$ such that $1<j<k\leq n$.
\end{lemma}

\begin{proof}
Let
$$[D_j]_j=\overline{x_{j-1}}+x_1f_1(X_{j-1})+f_2(X_{j-1}),$$
where $f_1$ and $f_2$ do not contain the variable $x_1$ and $f_1$ do not contain the monomial $\overline{x_{j-1}}/x_1$ and let
$$[D_k]_k=\overline{x_{k-1}}+x_1x_jg_1(X_{k-1})+x_1g_2(X_{k-1})+x_jg_3(X_{k-1})+g_4(X_{k-1}),$$
where $g_1,g_2,g_3,g_4$ do not contain the variables $x_1$ and $x_j$ and $g_1$ do not contain the monomial $\overline{x_{k-1}}/(x_1x_j)$.

\smallskip

Of course $[D_kD_jD_1D_kD_1D_j]_l=0$ for every $l=1,\ldots,n$ such that $l\neq k$. Let us investigate the~$k$-th coordinate:
$$[D_kD_jD_1D_kD_1D_j]_k=x_jh_1(X_{k-1})+h_2(X_{k-1}),$$
where
$$h_1(X_{k-1})=g_1(X_{k-1})+\overline{x_{k-1}}/(x_1x_j)$$
and
\begin{align*}
h_2(X_{k-1})=&(\overline{x_{k-1}}/(x_1x_j))\Big(x_1f_1(X_{j-1})+f_2(X_{j-1})\Big)+\\
&+x_1\Big(f_2(X_{j-1})g_1(X_{k-1})+f_1(X_{j-1})g_3(X_{k-1})\Big)+\\
&+g_2(X_{k-1})+f_2(X_{j-1})g_1(X_{k-1})+f_2(X_{j-1})g_3(X_{k-1}).
\end{align*}
Let us notice that $h_1$ and $h_2$ do not contain the variable $x_j$. Thus $[D_kD_jD_1D_kD_1D_j]_3=0$ if and only if $h_1(X_{k-1})=0$ and $h_2(X_{k-1})=0$. But from assumption that $g_1$ do not contain the monomial $\overline{x_{k-1}}/(x_1x_j)$ we get $h_1(X_{k-1})\neq0$ and hence $[D_kD_jD_1D_kD_1D_j]_k\neq0$.
\end{proof}

\begin{proposition}\label{prop:3}
A diagonal base $D$ such that the girth of $\Cay(P_n(2),D)$ is equal to 6 exists if and only if $n\geq5$.
\end{proposition}

\begin{proof}
Firstly, let us show that for $n<5$ there are no Cayley graphs of the group $P_n(2)$ on diagonal base $D$ which have girth equal to 6.

\smallskip

Graph $\Cay(P_2(2),D)$ for any diagonal base $D$ is isomorphic to cyclic graph $C_8$, so its girth is equal to 8.

\smallskip

Graph $\Cay(P_3(2),D)$, where $D=\{D_1,D_2,D_3\}$ is a diagonal base of $P_3(2)$, contains a~minimal cycle of length 6 only if the equation
$$D_3D_2D_1D_3D_1D_2=0$$
holds, but from Lemma \ref{l:c6} this is not possible.

\smallskip

Graph $\Cay(P_4(2),D)$, where $D=\{D_1,D_2,D_3,D_4\}$ is a diagonal base of $P_4(2)$, contains a~minimal cycle of length 6 if and only if none of elements of $D$ commute and the equation
$$D_4D_3D_2D_4D_2D_3=0$$
holds. Direct calculations show that there is no diagonal base $D$ satisfying both conditions.

\bigskip

Now let us recall the following base $S=\{S_1,\ldots,S_5\}$ of $P_5(2)$:

\bigskip

\begin{tabular}{llllll}
$S_1=[1,$&$0,$&$0,$&$0,$&$0$&$],$\\
$S_2=[0,$&$x_1+1,$&$0,$&$0,$&$0$&$],$\\
$S_3=[0,$&$0,$&$\overline{x_2}+x_2,$&$0,$&$0$&$],$\\
$S_4=[0,$&$0,$&$0,$&$\overline{x_3}+x_2x_3+x_1+1,$&$0$&$],$\\
$S_5=[0,$&$0,$&$0,$&$0,$&$\overline{x_4}+x_2x_3+x_3x_4+x_2+x_4$&$],$\\
\end{tabular}

\bigskip

Let us notice, that for any $1\leq i<j\leq5$ the inequality $(S_jS_i)^2\neq0$ holds. Moreover it is easy to check that
$$S_5S_4S_2S_5S_2S_4=0,$$
thus the girth of $\Cay(P_5(2),S)$ is equal to 6.

Applying Proposition \ref{prop:2} to this base we obtain a Cayley graph $\Cay(P_n(2),D)$ with girth equal to 6 for any $n>5$.

\end{proof}

Let us recall the diagonal base from Example \ref{e:1}, which do not contain a cycle of length 4.

\begin{proposition}\label{prop:4}
Let $D=\{D_1,\ldots,D_n\}$ be a diagonal base of $P_n(2)$, $n\geq2$, such that $[D_1]_1=1$ and $[D_i]_i=\overline{x_{i-1}}$ for every $i=2,\ldots,n$. The girth of graph $\Cay(P_n(2),D)$ is equal to 8.
\end{proposition}

\begin{proof}
From Proposition \ref{prop:1} we get that graph $\Cay(P_n(2),D)$ do not contain a cycles of length~4. Graph $\Cay(P_n(2),D)$ may contain a cycle of length 6 only if $n\geq3$ and it will contain such cycles only if for some integers $i,j,k$ such that $i<j<k\leq n$ the equation (\ref{eq:1}) holds. Obvieusly $[D_kD_jD_iD_kD_iD_j]_i=0$ and $[D_kD_jD_iD_kD_iD_j]_j=0$. Thus we have to prove that $[D_kD_jD_iD_kD_iD_j]_k\neq0$. Let us consider two cases:
\begin{enumerate}

\item $i=1$.

From Lemma \ref{l:c6} the inequality $D_kD_jD_1D_kD_1D_j\neq0$ holds. Moreover let us notice that this case shows that if $n=3$ then graph $\Cay(P_n(2),D)$ do not contain a cycle of length~6.

\item $i\geq2$.

Let us notice that this case is possible only for $n\geq4$.
\begin{align*}
[D_kD_jD_iD_kD_iD_j]_k&=\overline{x_{k-1}}+(\overline{x_{k-1}}/(x_ix_j))(x_i+\overline{x_{i-1}})(x_j+\overline{x_{j-1}})=\\
&=\overline{x_{k-1}}/x_i\neq0.
\end{align*}
\end{enumerate}
Thus inequality $[D_kD_jD_iD_kD_iD_j]_k\neq0$ holds in both cases. Hence the graph $\Cay(P_n(2),D)$ do not contain cycles of length 4 and 6, and so the girth of this graph is equal to 8 for all $n\geq2$.
\end{proof}

\bigskip

\textbf{Proof of Main Theorem.}

\bigskip

The statements 2, 3 and 4 hold by  Prospositions \ref{prop:1}, \ref{prop:3} and \ref{prop:4}. We are left to show statement~1 of the Main Theorem.

As we stated at the beginning of this Section, every cycles of $\Cay(P_n(2),D)$ have an even length and every such graph contains induced subgraph isomorphic to the cyclic graph $C_8$. Thus the girth of such graph is even and not greater then 8.  Example \ref{e:1} and Proposition \ref{prop:4} provides Cayley graphs of $P_n(2)$ on diagonal bases with girths equal to 4 and 8. Proposition \ref{prop:3} provides provides Cayley graphs of $P_n(2)$ on diagonal bases with girth equal to 6.

\section{Non-isomorphic Cayley graphs of $P_n(2)$ on diagonal bases.}

Throughout this Section we assume that $n\geq3$.

Let us recall diagonal bases of group $P_n(2)$ proposed in Example \ref{e:1}. Let $D=\{D_1,\ldots,D_n\}$ be a~diagonal base of group $P_n(2)$ such that $[D_1]_1=1$, $[D_n]_n=\overline{x_{n-1}}$ and for every $i=2,\ldots,n-1$ non-zero coordinates of tableaux $D_i$ are of form
$$[D_i]_i=\overline{x_{i-1}}+a_i,$$ where $a_i\in\mathbb{Z}_2$. From now on we will call such diagonal base of $P_n(2)$ as \textit{a $\delta$-base}.

\smallskip

Every $\delta$-base $D$ of $P_n(2)$ can be uniquely identified with a binary vector $a_D=[a_2,\ldots,a_{n-1}]$, called \textit{the characteristic vector} of a $\delta$-base $D$.

\smallskip

Let us notice that if $a_D=[0,0,\ldots,0]$, then from Proposition \ref{prop:4} the girth of $\Cay(P_n(2),D)$ is equal to 8  and if $a_D\neq[0,0,\ldots,0]$, then from Example \ref{e:1} the girth of $\Cay(P_n(2),D)$ is equal to 4. Thus the Cayley graph $\Cay(P_n(2),D)$, where $a_D=[0,\ldots,0]$ is not isomorphic with any Cayley graph $\Cay(P_n(2),D')$, where $a_{D'}\neq[0,\ldots,0]$. Moreover, the following Theorem holds.

\begin{theorem}\label{nis}
For different $\delta$-bases $D$ and $D'$ of $P_n(2)$ the Cayley graphs $\Cay(P_n(2),D)$ and $\Cay(P_n(2),D')$ are not isomorphic.
\end{theorem}

\begin{proof}
Let $D=\{D_1,\ldots,D_n\}$ and $D'=\{D_1',\ldots,D_n'\}$ be different $\delta$-bases of $P_n(2)$ with characteristic vectors $a_D=[a_2,\ldots,a_{n-1}]$ and $a_{D'}=[a_2',\ldots,a_{n-1}']$ respectively.

From Proposition \ref{prop:1} we get that if $a_j=0$ then $D_j$ do not commute with any $D_m$ for $m>j$, and if $a_j=1$ then $D_j$ commutes with every $D_m$ for $m>j$. Let us recall that $D_1$ do not commute with any other elements of $D$. Thus the number of pairs of commuting elements of $\delta$-base $D$ is equal to \begin{equation}\label{e:2}\displaystyle \sum\limits_{i=2}^{n-1}a_i(n-i).
\end{equation}
Further we denote the number of pairs of commuting elements of $\delta$-base $D$ by $c_D$.

Every tableau $f\in P_n(2)$ is a vertex of the Cayley graph $\Cay(P_n(2),D)$. Let us notice that the~vertex~$f$ belongs exactly to $c_D$ different cycles of length 4 (every pair of commuting generators $D_\alpha,\,D_\beta$ gives an unique cycle of length 4 of form $f\leftrightarrow fD_\alpha\leftrightarrow fD_\alpha D_\beta\leftrightarrow fD_\beta\leftrightarrow f$). Thus if $c_D\neq c_{D'}$, then Cayley graphs $\Cay(P_n(2),D)$ and $\Cay(P_n(2),D')$ are not isomorphic.

\smallskip

Now let us investigate the case $c_D=c_{D'}$.

Firstly, let us determine the form of ball $B_D(f,2)$ of the graph $\Cay(P_n(2),D)$. Let us notice that the vertex set $V(B_D)$ of ball $B_D(f,2)$ is equal to
$$V(B_D)=\left\{g\in P_n(2)\,|\,\exists(i,j\in\{1,\ldots,n\}):\,g=f\,\vee\,g=fD_i\,\vee\,g=fD_iD_j\right\}.$$
For some $i$ $(i=1,\ldots,n)$ let $f'=fD_i$. If element $D_i$ commutes with some other element of $\delta$-base $D$, then edge $ff'$ of graph $B_D(f,2)$ is contained in some cycle of length 4. On the other hand, if $D_i$ does not commute with any other element of $D$ in graph $B_D(f,2)$, then edge $ff'$ is not contained in any cycle and, hence, for every $D_j\neq D_i$ the path of form
$$f\leftrightarrow ff'\leftrightarrow ff'D_j$$
is not contained in any cycle. Thus the structure of induced subgraph $B_D(f,2)$ of graph $\Cay(P_n(2),D)$ is as follows:

Vertex $f$ is contained in $c_D$ number of cycles of length 4, and for every noncommuting element $D_i$ of $D$ vertex $f$ is contained in the path $$f\leftrightarrow fD_i\leftrightarrow fD_iD_j,$$ which is not contained in any cycle.

Examples of induced subgraphs $B_D(f,2)$ of $\Cay(P_5(2),D)$ are presented in Figure~\ref{fig1}.

\smallskip

Let us notice that the structure of $B_D(f,2)$ does not depend on the vertex $f$, so for any $g\in P_n(2)$ graph $B_D(g,2)$ is isomorphic to $B_D(f,2)$. Hence if graphs $B_D(f,2)$ and $B_{D'}(f,2)$ are not isomorphic, then also graphs $\Cay(P_n(2),D)$ and $\Cay(P_n(2),D')$ are not isomorphic.

\smallskip

We assumed that $c_D=c_D'$ and that bases $D$ and $D'$ are different. Thus vectors $a_D$ and $a_{D'}$ are also different. Let $k$ be the least integer such that $a_k\neq a_k'$. We may assume that $a_k=1$ and $a_k'=0$. Thus the number of elements of $D$ which commutes with $D_k$ is equal to $$n-k+\sum\limits_{i=2}^{k-1}a_i.$$ Let us denote that number by $\theta$. Let $\Theta$ be the number of generators $D_i$, where $i<k$,  which commute with exactly $\theta$ generators from $D$. Thus the number of all generators of $D$ which commute with exactly $\theta$ generators is not less then $\Theta+1$. On the other hand, from the fact that for all $i\leq k$ equality $a_i=a_i'$ holds, the number of generators of $D'$ which commutes with exactly $\theta$ other generators from $D'$ is equal to $\Theta$ (notice that there is no generator $D_s'$ for $s>k$, which commute with exactly $\theta$ other generators). Thus the number of edges of the form $f\leftrightarrow fD_i$ which are contained in exactly $\theta$ cycles of length 4 in graph $B_D(f,2)$ is not equal to the number of edges of the form $f\leftrightarrow fD_i'$ which are contained in exactly $\theta$ cycles of length 4 in graph $B_{D'}(f,2)$, and hence these graphs are not isomorphic.


\end{proof}

\begin{figure}[h]
    \centering
    \includegraphics[width=0.8\hsize]{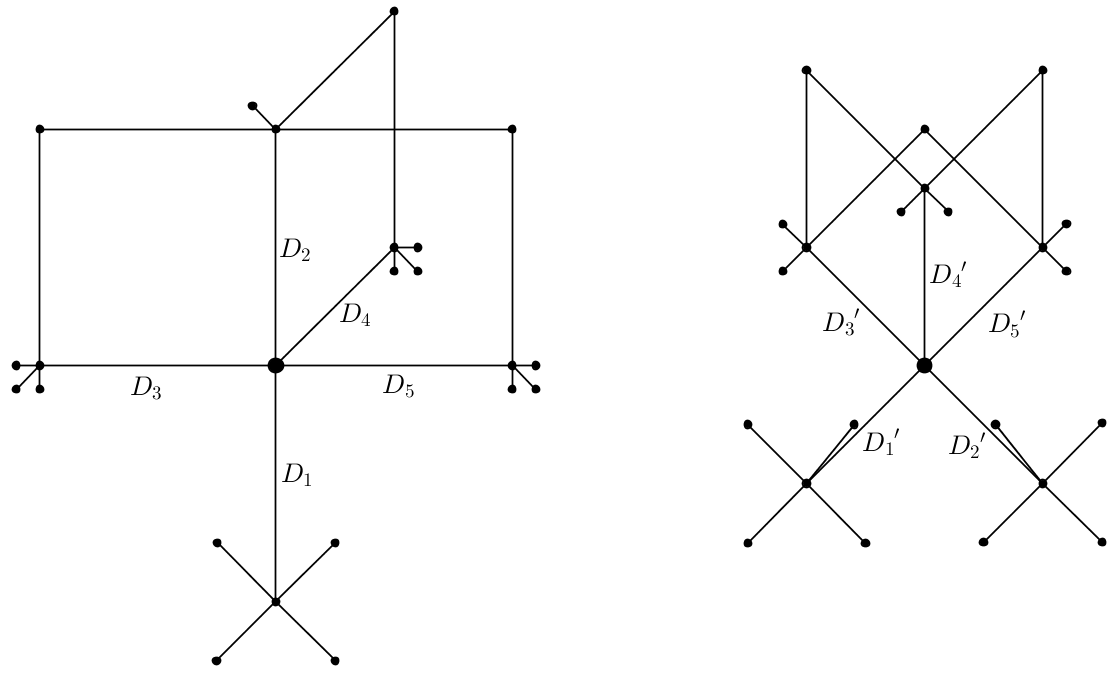}
    \caption{Graphs $B_D(f,2)$ and $B_{D'}(f,2)$ where $D$ and $D'$ are $\sigma$-bases with characteristic vectors $a_D=[1,0,0]$ and $a_{D'}=[0,1,1]$.}
    \label{fig1}
\end{figure}

Let us notice, that both bases $D$ and $D'$ from Figure \ref{fig1} have exactly three pairs of commuting elements, but graphs $B_D(f,2)$ and $B_{D'}(f,2)$ are not isomorphic.

\bigskip

From Theorem \ref{nis} the number of nonisomorphic Cayley graphs of $P_n(2)$ on the diagonal bases is at least equal to the number of all $\delta$-bases of this group. The number of all $\delta$-bases of $P_n(2)$ is equal to the number of all different characteristic vectors of those bases, so there holds the following

\begin{corollary}
The number of non-isomorphic Cayley graphs of $P_n(2)$ on diagonal bases is not less then $2^{n-2}$.
\end{corollary}

\end{document}